\documentclass{article}
\usepackage{amsfonts}
\usepackage{amssymb}
\usepackage{graphicx}
\usepackage{amsmath}
\usepackage[singlespacing]{setspace}
\usepackage{fancyhdr}

\newtheorem{theorem}{Theorem}[section]

\newtheorem{condition}[theorem]{Condition}

\newtheorem{definition}[theorem]{Definition}
\newtheorem{example}[theorem]{Example}

\newtheorem{lemma}[theorem]{Lemma}
\newtheorem{notation}[theorem]{Notation}

\newtheorem{proposition}[theorem]{Proposition}
\newtheorem{remark}[theorem]{Remark}

\newenvironment{proof}[1][Proof]{\noindent\textbf{#1.} 
}{\ \rule{0.6em}{0.6em}}
\input tcilatex 
\begin{document}

\title{Reduced Gr\"{o}bner Bases of Certain Toric Varieties; A New Short
Proof}
\author{Ibrahim Al-Ayyoub}
\maketitle

\begin{abstract}
Let $K$ be a field and let $m_{0},...,m_{n}$ be an almost arithmetic
sequence of positive integers. Let $C$ \ be a toric variety in the affine $%
\left( n+1\right) $-space, defined parametrically by $x_{0}=t^{m_{0}},\ldots
,x_{n}=t^{m_{n}}$. In this paper we produce a minimal Gr\"{o}bner basis for
the toric ideal which is the defining ideal of $C$ and give sufficient and
necessary conditions for this basis to be the reduced Gr\"{o}bner basis of $%
C $ , correcting a previous work of \cite{Sen} and giving a much simpler
proof than that of \cite{Ayy}.
\end{abstract}

\section*{Introduction}

Let $n\geq 2$, $K$ a field and let $x_{0},\ldots ,x_{n},t$ be
indeterminates. Let $m_{0},\ldots ,m_{n}$ be an almost arithmetic sequence
of positive integers, that is, some $n-1$ of these form an arithmetic
sequence, and assume $gcd(m_{0},\ldots ,m_{n})=1$. Let $P$\ be the kernel of
the $K$-algebra homomorphism $\eta :K[x_{0},\ldots ,x_{n}]\rightarrow K[t]$,
defined by $\eta (x_{i})=t^{m_{i}}$. Such an ideal is called a \textit{toric
ideal }and the variety $V(P)$, the zero set of $P$, is called an \textit{%
affiine toric variety}. The definition of toric variety that we us is the
same as the definition given in \cite{Stu1}. This differs from the
definition found in the algebraic geometry literature (as in \cite{Fulton})
which requires the variety to be normal. Toric ideals are an interesting
kind of ideals that have been studied by many authors, for example, see \cite%
{Stu2} and\ Chapter 4 of \cite{Stu1}. The theory of toric varieties plays an
important role at the crossroads of geometry, algebra and combinatorics.

\ \ \ 

A set of generators for the ideal $P$ was explicitly constructed in \cite%
{PaSi}. We call these generators the \textit{Patil-Singh generators.} Out of
this generating set, Patil \cite{Pat} constructed a minimal generating set $%
\Omega $ for the ideal $P$. We call the elements of $\Omega ~$the \textit{%
Patil generators. }Sengupta \cite{Sen} proved that $\Omega $ forms a Gr\"{o}%
bner basis for the relation ideal $P$ with respect to the grevlex monomial
order, however, Al-Ayyoub \cite{Ayy} showed that Sengupta's proof is not
complete, as in fact $\Omega $ is not a Gr\"{o}bner basis in all cases, see
Remark~\ref{Patil-Not-Gb} and Remark~\ref{PnotGBinSomeA3}. The proof
introduced by Al-Ayyoub \cite{Ayy} is computational as it uses the
Buchberger criterion and the division algorithm and it did not characterize
whether the given Gr\"{o}bner basis is reduced. The goal of this paper is to
produce a minimal Gr\"{o}bner basis for $P$, give sufficient and necessary
conditions for this basis to be reduced, and to give a new proof that is
based on a lemma of Aramova et al. \cite{AHH}. The proof given in this paper
is much shorter and simpler than the computational work given in \cite{Ayy}
or \cite{Sen}. The author thanks the referee for suggesting to use a result
of \cite{AHH} that shortened the proof.

\section{Generators for Toric Varieties\label{Patil-Gens}}

In this part we recall the construction, given in \cite{PaSi} and \cite{Pat}%
, of the generating set of the defining ideal $P$ of certain monomial curves
(toric varieties), and we also recall the result of \cite{Ayy} proving that
the set given in \cite{Pat}\ is not a Gr\"{o}bner basis for $P$. We shall
use the notation and the terminology from \cite{PaSi} and \cite{Pat} with a
slight difference in naming some variables and constants. Let $n\geq 2$ be
an integer and let $p=n-1$ . Let $m_{0},\ldots ,m_{p}$ be an arithmetic
sequence of positive integers with $0<m_{0}<\cdots <m_{p}$, let $m_{n}$ be
arbitrary, and $gcd(m_{0},\ldots ,m_{n})=1$. Let $\Gamma $ denote the
numerical semigroup that is generated by $m_{0},\ldots ,m_{n}$ i.e. $\Gamma
=\sum\limits_{i=0}^{n}\mathbb{N}_{0}m_{i}$ with $\mathbb{N}_{0}=\{0\}\cup 
\mathbb{N}$. We assume throughout that $\Gamma $\ is minimally generated by $%
m_{0},\ldots ,m_{n}$. Put $\Gamma ^{\prime }=\sum\limits_{i=0}^{p}\mathbb{N}%
_{0}m_{i}$. Thus $\Gamma =\Gamma ^{\prime }+\mathbb{N}_{\mathbf{0}}m_{n}$.
Let $S=\{\gamma \in \Gamma \mid \gamma -m_{0}\notin \Gamma \}$.

\begin{notation}
\label{g-sub-t}\textit{For }$a,b\in \mathbb{Z}~$\textit{\ let }$[a,b]=\{t\in 
\mathbb{Z}\mid a\leq t\leq b\}$\textit{. For }$t\geq 0$\textit{, let }$%
q_{t}\in \mathbb{Z}\mathbf{,}$\textit{\ }$r_{t}\in \lbrack 1,p]$\textit{\
and }$g_{t}\in \Gamma ^{\prime }$\textit{\ \ be defined by }$t=q_{t}p+r_{t}$%
\textit{\ and }$g_{t}=q_{t}m_{p}+m_{r_{t}}$.\ \ \ \ \ \ \ \ \ \ \ \ \ 
\end{notation}

The following lemma provides us with the parameters and the equalities that
are crucial for the new proof.\ \ 

\begin{lemma}
\label{Parameters}(Lemmas 3.1 and 3.2, \cite{PaSi})\textbf{\ }\textit{Let }$%
u=min\{t\geq 0\mid g_{t}\notin S\}$\textit{\ and }$\upsilon =min\{b\geq
1\mid bm_{n}\in \Gamma ^{\prime }\}$\textit{.}\newline
(a) \textit{There exist unique integers }$w\in \lbrack 0,\upsilon -1]$%
\textit{, }$z\in \lbrack 0,u-1]$\textit{, }$\lambda \geq 1$\textit{, }$\mu
\geq 0$\textit{, and }$\nu \geq 2$\textit{\ such that\newline
\ \ (i) }$g_{u}=\lambda m_{0}+wm_{n}$\textit{;\newline
\ \ (ii) }$\upsilon m_{n}=\mu m_{0}+g_{z}$\textit{;\newline
\ \ (iii) }$g_{u-z}+(\upsilon -w)m_{n}=\nu m_{0}$\textit{, where }$\nu
=\left\{ 
\begin{tabular}{ll}
$\lambda +\mu +1\text{,}$ & if$\text{\ \ }r_{u-z}<r_{u}\text{;}$ \\ 
$\lambda +\mu \text{,}$ & $\text{if \ }r_{u-z}\geq r_{u}\text{.}$%
\end{tabular}%
\right. $\newline
\newline
\textit{(b) Let }$\ V=[0,u-1]\times \lbrack 0,\upsilon -1]$\textit{\ and }$%
W=[u-z,u-1]\times \lbrack \upsilon -w,\upsilon -1]$\textit{. Then every
element of }$\Gamma $ can be expressed uniquely in the form $%
am_{0}+g_{s}+bm_{n}$ with $a\in \mathbb{N}_{0}$ and $(s,b)\in V-W.$
\end{lemma}

\begin{notation}
\label{W&q'&r'}\textit{Let }$q=q_{u},$\textit{\ }$r=r_{u},$\textit{\ }$%
q^{\prime }=q_{u-z},$\textit{\ }$r^{\prime }=r_{u-z}$\textit{. From now on,
the symbols }$q,$\textit{\ }$q^{\prime },$\textit{\ }$r,$\textit{\ }$%
r^{\prime },$\textit{\ }$u,\upsilon ,$\textit{\ }$w,$\textit{\ }$z,$\textit{%
\ }$\lambda ,$\textit{\ }$\mu ,$\textit{\ }$\nu ,V$\textit{\ and }$W$\textit{%
\ will have the meaning assigned to them by this notation and the lemma
above.}
\end{notation}

\begin{remark}
\label{q>0-and-u>p}Note that for $1\leq i\leq p$ we have $%
g_{i}-m_{0}=m_{i}-m_{0}.$ \textit{Then by the minimality assumption on the
generators of \ }$\Gamma $\textit{\ it follows that} $u>p$, hence $q>0.$
\end{remark}

We recall the construction and the result given in \cite{PaSi}: let $p=n-1$
and let

\ \ \ \ \ 

$\xi _{i,j}=\left\{ 
\begin{tabular}{ll}
$x_{i}x_{j}-x_{0}x_{i+j}$, & $\text{if\ \ \ \ }i+j\leq p$; \\ 
$x_{i}x_{j}-x_{i+j-p}x_{p}$, & $\text{if \ \ \ }i+j>p$,%
\end{tabular}%
\right. $

$\varphi _{i}=x_{r+i}x_{p}^{q}-x_{0}^{\lambda -1}x_{i}x_{n}^{w}$,

$\psi _{j}=x_{r^{\prime }+j}x_{p}^{q^{\prime }}x_{n}^{\upsilon
-w}-x_{0}^{\nu -1}x_{j}$,

$\theta $\ $=\left\{ 
\begin{tabular}{ll}
$x_{n}^{\upsilon }-x_{0}^{\mu }x_{r-r^{\prime }}x_{p}^{q-q^{\prime }}\text{,}
$ & if $\text{\ }r^{\prime }<r\text{;}$ \\ 
$x_{n}^{\upsilon }-x_{0}^{\mu }x_{p+r-r^{\prime }}x_{p}^{q-q^{\prime }-1}%
\text{,}$ & if $\text{\ }r^{\prime }\geq r$.%
\end{tabular}%
\right. $\newline

\ \ \ \ 

The following intervals are introduced by \cite{Pat} in the process of
producing minimal generating sets.

\ \ \ \ \ \ \ 

$I=\left\{ 
\begin{tabular}{ll}
$\lbrack 0,p-r]\text{,}$ & if$\text{ }\mu \neq 0\text{ }$or$\text{ }W=\phi 
\text{;}$ \\ 
$\lbrack \max (r_{z}-r+1,0),p-r]\text{,}$ & if$\text{ }\mu =0\text{ }$and$%
\text{ }W\neq \phi $,%
\end{tabular}%
\ \right. $

$J=\left\{ 
\begin{tabular}{ll}
$\phi \text{,}$ & if $\text{\ }W=\phi \text{;}$ \\ 
$\lbrack 0,\min (z-1,p-r^{\prime })]\text{,}$ & if$\text{ \ }W\neq \phi 
\text{.}$%
\end{tabular}%
\ \right. $

\begin{theorem}
\label{Patil&Patil-Singh Gens}(Theorem 4.5, \cite{PaSi}) The set 
\begin{equation*}
\{\xi _{i,j}\mid 1\leq i\leq j\leq p-1\}\cup \{\theta \}\cup \{\varphi
_{i}\mid 0\leq i\leq p-r\}\cup \{\psi _{j}\mid 0\leq j\leq p-r^{\prime }\}
\end{equation*}%
forms a generating set for the ideal $P$. The elements in this set are
called the \textit{Patil-Singh generators}. Also, (Theorem 4.5, \cite{Pat})
the set 
\begin{equation*}
\Omega =\{\xi _{i,j}\mid 1\leq i\leq j\leq p-1\}\cup \{\theta \}\cup
\{\varphi _{i}\mid i\in I\}\cup \{\psi _{j}\mid j\in J\}
\end{equation*}%
forms a minimal generating set for the ideal $P$. The elements in this set
are called the Patil \textit{generators}.
\end{theorem}

Considering the indices we note that handling the Patil-Singh generators is
simpler than the Patil generators.

Sengupta \cite{Sen} tried to prove that the set $\Omega $ forms a Gr\"{o}%
bner basis for $P$ with respect to the grevlex monomial order using the
grading $wt(x_{i})=m_{i}$ with\textit{\ }$x_{0}<x_{1}<\cdots <x_{n}$. In
this ordering \textit{\ }$\prod\limits_{i=0}^{n}x_{i}^{a_{i}}>_{grevlex}%
\prod\limits_{i=0}^{n}x_{i}^{b_{i}}$\ if in the ordered tuple $%
(a_{1}-b_{1},\ldots ,a_{n}-b_{n})$\ the left-most nonzero entry is negative.
Al-Ayyoub \cite{Ayy} proved that Sengupta's proof works for arithmetic
sequences, but it is incomplete for the almost arithmetic sequences. Below
we recall the work of \cite{Ayy} for the convenience of the reader;

\begin{remark}
\label{Patil-Not-Gb} Assume $r^{\prime }\geq r$ , $\mu =0$, and$\text{ }%
W\neq \phi $. Then \textit{Patil generators are not a Gr\"{o}bner basis with
respect to the grevlex monomial ordering with }$x_{0}<x_{1}<\cdots <x_{n}$ 
\textit{and\ with the grading }$wt(x_{i})=m_{i}$\textit{.}
\end{remark}

\begin{proof}
As $u-z=(q-q_{z})p+(r-r_{z})$ then $r^{\prime }\geq r$ if and only if $%
r_{z}\geq r$. Assume $r^{\prime }\geq r$, then $r_{z}-r+1>0$ and also $%
\theta =x_{n}^{\upsilon }-x_{0}^{\mu }x_{p+r-r^{\prime }}x_{p}^{q-q^{\prime
}-1}$. Assume also that $\mu =0\text{ }$and$\text{ }W\neq \phi $, then $%
I=[\max (r_{z}-r+1,0),p-r]=[r_{z}-r+1,p-r]$. Under these assumptions the
S-polynomial $S(\psi _{k},\theta )$\ can not be reduced to zero modulo $%
\Omega $: for $0\leq k<r_{z}-r+1$ consider $S(\psi _{k},\theta )=x_{0}^{\mu
}S_{1}$ where $S_{1}=x_{0}^{\lambda -1}x_{k}x_{n}^{w}-\underline{%
x_{r^{\prime }+k}x_{p+r-r^{\prime }}x_{p}^{q-1}},$ with the leading monomial
underlined. We note that $LM(S_{1})$, the leading monomial of $S_{1}$, is a
multiple of $LM(\xi _{r^{\prime }+j,p+r-r^{\prime }})$ only. Hence, the only
possible way to reduce $S_{1}$ with respect to $\Omega $ is by using $\xi
_{r^{\prime }+j,p+r-r^{\prime }}$. However, none of the terms of the
binomial $S_{1}+x_{p}^{q-1}\xi _{r^{\prime }+j,p+r-r^{\prime
}}=x_{r+k}x_{p}^{q}-x_{0}^{\lambda -1}x_{k}x_{n}^{w}$ is a multiple of any
of the leading terms of Patil generators. Therefore, it can not be reduced
to $0$ modulo $\Omega $.
\end{proof}

\ \ \ \ 

The following shows that the hypothesis of the remark above are satisfied by
an infinite family of toric varieties:

\begin{remark}
\label{PnotGBinSomeA3} Let $m_{0}\geq 5$ be an odd integer. Let $P$ be the
defining ideal of the toric variety that corresponds to the almost
arithmetic sequence $m_{0},m_{0}+1,m_{0}-1$. Then the Patil generators for
the ideal $P$ are not\ a Gr\"{o}bner basis \textit{with respect to the
grevlex monomial ordering with }$x_{0}<x_{1}<x_{2}\ $\textit{and with the
grading }$wt(x_{i})=m_{i}$\textit{.}
\end{remark}

\begin{proof}
Observe: $p=1,n=2$ , and $g_{i}=i(m_{0}+1)$ for all $i$.

Let $\upsilon ,\mu ,$ and $z$\ be as defined in Lemma~\ref{Parameters}. Then 
$\upsilon (m_{0}-1)=\mu m_{0}+z(m_{0}+1)$ for some integers $\mu ,z\geq 0$ .
This implies $\mu +z<v.$ Note that $\upsilon (m_{0}-1)=\mu
m_{0}+z(m_{0}+1)=(\mu +z)(m_{0}-1)+\mu +2z$. Thus $\mu +2z=s(m_{0}-1)$ for
some $s\geq 1$. Hence, $\upsilon >\mu +z\geq \dfrac{\mu }{2}+z=\dfrac{s}{2}%
(m_{0}-1)\geq \frac{m_{0}-1}{2}$. Thus,%
\begin{equation}
\upsilon \geq \frac{m_{0}+1}{2}.  \label{v>or=(m0+1)/2}
\end{equation}%
On the other hand, note that%
\begin{equation}
\dfrac{m_{0}+1}{2}(m_{0}-1)=\dfrac{m_{0}-1}{2}(m_{0}+1)\in \Gamma ^{\prime }.
\label{v((m0-1))}
\end{equation}%
Therefore, by the minimality of $\upsilon $ we must have%
\begin{equation}
\upsilon \leq \frac{m_{0}+1}{2}.  \label{v<=(m0+1)/2}
\end{equation}%
By~$\left( \ref{v>or=(m0+1)/2}\right) $ and ~$\left( \ref{v<=(m0+1)/2}%
\right) $ it follows that $\upsilon =\dfrac{m_{0}+1}{2}$.

Let $u,\lambda ,w,$ and $g_{u}$ be as defined in Lemma~\ref{Parameters}. Note%
\begin{equation}
\dfrac{m_{0}+1}{2}(m_{0}+1)-m_{0}=\dfrac{m_{0}-1}{2}(m_{0}-1)+m_{0}\in
\Gamma .  \label{u((m0+1))}
\end{equation}%
Therefore,%
\begin{equation}
u\leq \dfrac{m_{0}+1}{2}.  \label{u<or=(m0+1)/2}
\end{equation}

Claim $w>0$: if $w=0$ then $g_{u}=\lambda m_{0}$, thus $u(m_{0}+1)=\lambda
m_{0}$. But $m_{0}$ and $m_{0}+1$ are relatively prime, therefore, we must
have $u=bm_{0}$ for some $b\geq 1$, a contradiction to~$\left( \ref%
{u<or=(m0+1)/2}\right) $. Thus $w>0$.

Claim $\lambda <u$: by Lemma~\ref{Parameters} we have $u(m_{0}+1)=\lambda
m_{0}+w(m_{0}-1)$. If $\lambda \geq u$ then $w(m_{0}-1)=u(m_{0}+1)-\lambda
m_{0}=u+(u-\lambda )m_{0}$, which implies $u\geq m_{0}-1$ as $w>0$, a
contradiction to~$\left( \ref{u<or=(m0+1)/2}\right) $. Thus $\lambda <u$.

Now consider $w(m_{0}-1)=u(m_{0}+1)-\lambda m_{0}=\left( u-\lambda \right)
(m_{0}-1)+2u-\lambda $. As $w(m_{0}-1)>0$ and $u>\lambda $ we must have $%
2u-\lambda =c(m_{0}-1)$ for some $c\geq 1$. But if $u\leq \dfrac{m_{0}-1}{2}$
then $2u-\lambda \leq m_{0}-1-\lambda $, a contradiction as $\lambda \geq 1$%
. Therefore, 
\begin{equation}
u>\dfrac{m_{0}-1}{2}.  \label{u>>(m0-1)/2}
\end{equation}%
By $\left( \ref{u<or=(m0+1)/2}\right) $ and $\left( \ref{u>>(m0-1)/2}\right) 
$ it follows that $u=\dfrac{m_{0}+1}{2}$.

Now by the uniqueness in Lemma~\ref{Parameters} and as of $\left( \ref%
{v((m0-1))}\right) $ and $\left( \ref{u((m0+1))}\right) $ it follows that $%
\mu =0$, $z=\dfrac{m_{0}-1}{2}$, $\lambda =2$ and $w=\dfrac{m_{0}-1}{2}$.
Finally, note that $r=p=r^{\prime }=1$. Therefore, the parameters $z,w,\mu
,p,r,$ and $r^{\prime }$ all satisfy the assumptions of the previous remark,
hence done.
\end{proof}

\section{Reduced Gr\"{o}bner Bases\label{PtoGB}\ }

In the following we combine the results of \cite{PaSi}\ and \cite{Pat} to
obtain the set of generators that we prove to be a minimal (the reduced) Gr%
\"{o}bner Basis. In particular, we pick an appropriate set of indices
(different from Sengupta \cite{Sen}), as well as, we modify the form of the
binomial $\theta $ as follows; let $u,$ $z,$ $q,$ $r,$ $q^{\prime }=q_{u-z}$%
, and $r^{\prime }=r_{u-z}$ be as in Lemma~\ref{Parameters} and Notation~\ref%
{W&q'&r'}. Let $z=q_{z}p+r_{z}$ with $q_{z}\in \mathbb{Z}$ and $r_{z}\in
\lbrack 1,p]$. By Notation~\ref{g-sub-t} it is clear that $q_{z}\leq q$
since $0\leq z\leq u-1$. As $u-z=(q-q_{z})p+(r-r_{z})$, it follows that $%
q^{\prime }=q-q_{z}-\varepsilon $ and $r^{\prime }=\varepsilon p+r-r_{z}$
where $\varepsilon =0$ or $1$ according as $r>r_{z}$ or $r\leq r_{z}.$
Therefore, $r^{\prime }<r$\ if and only if $r_{z}<r$. Thus we rewrite $%
\theta =x_{n}^{\upsilon }-x_{0}^{\mu }x_{r_{z}}x_{p}^{q_{z}}.$ Then the
generators that we prove to be a minimal (the reduced)\ Gr\"{o}bner basis
are as follows (with the leading monomial underlined);

\ \ 

\begin{tabular}{lll}
$\varphi _{i}$ & $=\underline{x_{r+i}x_{p}^{q}}-x_{0}^{\lambda
-1}x_{i}x_{n}^{w}$, & for \ \ $\ 0\leq i\leq p-r$; \ \  \\ 
$\psi _{j}$ & $=\underline{x_{r^{\prime }+j}x_{p}^{q^{\prime
}}x_{n}^{\upsilon -w}}-x_{0}^{\nu -1}x_{j}$, & for \textit{\ \ \ }$j\in J$;
\\ 
$\theta $ & $=\underline{x_{n}^{\upsilon }}-x_{0}^{\mu
}x_{r_{z}}x_{p}^{q_{z}}$, &  \\ 
$\xi _{i,j}$ & $=\left\{ 
\begin{tabular}{ll}
$\underline{x_{i}x_{j}}-x_{0}x_{i+j}$, & $\text{if\ \ \ \ }i+j\leq p$; \\ 
$\underline{x_{i}x_{j}}-x_{i+j-p}x_{p}$, & $\text{if \ \ \ }i+j>p$,%
\end{tabular}%
\right. $ & for\ \ \ $1\leq i\leq j\leq p-1$.%
\end{tabular}%
\newline

\ \ 

Note that this set of generators contains the set of Patil generators and it
is contained in the set of Patil-Singh generators.

\begin{definition}
\label{minimalGB}\textit{Let }$I$\textit{\ be a polynomial ideal and }$G$%
\textit{\ a Gr\"{o}bner basis for }$I$\textit{\ such that: \newline
(i) }$LC(f)=1$\textit{\ for all }$f\in G$\textit{, where }$LC(f)$\textit{\
is the leading coefficient of }$f$\textit{.\newline
(ii) For all }$f\in G$\textit{, }$LM(f)\notin \langle LM\{G-\{f\}\}\rangle .$%
\newline
(ii' ) \textit{For all }$f\in G$\textit{, no monomial\ appearing in }$f$%
\textit{\ lies in }$\langle LM\{G-\{f\}\}\rangle .$\newline
Then $G$ is called \textit{\textbf{minimal}} if it satisfies (i) and (ii),
and it is called \textbf{reduced} if it satisfies (i) and (ii').
\end{definition}

\begin{condition}
Let C1 and C2 refer to the conditions as follows

\textbf{C1}: $J\neq \phi ,q^{\prime }=0,\upsilon -w\leq w,\lambda =1,$ and $%
r^{\prime }\leq p-r$.

\textbf{C2}: $q=1$ and $r\leq p-2$.
\end{condition}

The following is the main result of this paper.

\begin{theorem}
\label{MainThm}\textit{The set }%
\begin{equation*}
G=\mathit{\ }\{\varphi _{i}\mid 0\leq i\leq p-r\}\cup \{\psi _{j}\mid j\in
J\}\cup \{\theta \}\mathit{\ }\cup \mathit{\ }\{\xi _{i,j}\mid 1\leq i\leq
j\leq p-1\}
\end{equation*}%
\textit{is a minimal Gr\"{o}bner basis for the ideal }$P$\textit{\ with
respect to the grevlex monomial order with }$x_{0}<x_{1}<\cdots <x_{n}$%
\textit{\ and with the grading }$wt(x_{i})=m_{i}$\textit{.}\ Moreover, $G$\
is reduced if and only if none of the conditions C1 and C2 holds.\ \ \ 
\end{theorem}

\begin{proof}
The proof that $G$ is a Gr\"{o}bner basis is after Lemma \ref{AH} below.
Here we prove that $G$ is minimal (or reduced).

It is clear that $LM(\theta )\notin \langle LM(G-\{\theta \})\rangle $.
Since $w<\upsilon $ (by Lemma~\ref{Parameters} ) and since $q>0$ (by Remark~%
\ref{q>0-and-u>p}) it is clear that $LM(\varphi _{i})\notin \langle
LM(G-\{\varphi _{i}\})\rangle $ and $LM(\xi _{i,j})\notin \langle LM(G-\{\xi
_{i,j}\})\rangle $. To show $LM(\psi _{j})\notin \langle LM(G-\{\psi
_{j}\})\rangle $ it is clear that it suffices to show that $LM(\psi _{j})$
is not a multiple of any of $LM(\varphi _{i})$. If $q_{z}>0$ or $\varepsilon
>0$, then this is clear since $q^{\prime }<q$ (as $q^{\prime
}=q-q_{z}-\varepsilon )$ and since $\upsilon -w<\upsilon $ whenever $J\neq
\phi $. If $q_{z}=0$ and $\varepsilon =0,$ then $r^{\prime }=r-r_{z}$ and $%
z-1=r_{z}-1<p-r+r_{z}=p-r^{\prime }$. Thus there is no overlap between the
indices of the leading monomials of $\varphi _{i}$ and those of $\psi _{j}$.
This shows $G$ is minimal.

Define $SM(f)=f-LM(f)$ with $f$ a binomial. Recalling that $\nu \geq 2$ and $%
x_{0}$ divides no $LM(f)$ for any $f\in G$, it follows that $SM(\psi
_{j})\notin \langle LM(G-\{\psi _{j}\})\rangle $. Also, recalling that $%
w<\upsilon $ and $z<u$, it follows that $SM(\theta )\notin \langle
LM(G-\{\theta \})\rangle $ and $SM(\xi _{i,j})\notin \langle LM(G-\{\xi
_{i,j}\})\rangle $. If any of the parts of condition \textit{C1} does not
hold, then it follows that $SM(\varphi _{i})\notin \langle LM\{\psi
_{i}\};j\in J\rangle $ which suffices to show $SM(\varphi _{i})\notin
\langle LM(G-\{\varphi _{i}\})\rangle $. To show $SM(\xi _{i,j})\notin
\langle LM(G-\{\xi _{i,j}\})\rangle ,$ it is enough to show $SM(\xi
_{i,j})\notin \left\langle LM(\varphi _{k});0\leq k\leq p-r\right\rangle $
whenever $i+j>p$ because $w<\upsilon $ and $r>0$. But this clear if any of
the parts of condition \textit{C2} does not hold (recall $i+j-p\leq p-2$).
This proves that if none of \textit{C1} and \textit{C2} holds, then $G$ is
reduced.

Conversely, assume \textit{C1} holds. Then as $q^{\prime }=0$ and $\lambda
=1 $ then $LM(\psi _{0})=x_{r^{\prime }}x_{n}^{\upsilon -w}$. On the other
hand, since $r^{\prime }\leq p-r$ then $SM(\varphi _{r^{\prime
}})=x_{r^{\prime }}x_{n}^{w}$. Thus $SM(\varphi _{r^{\prime }})$ is a
multiple of \ $LM(\psi _{0})$\ whenever $\upsilon -w\leq w$. Thus $G$ is not
reduced. Assume \textit{C2} holds. Choose $i=p-1$ and $j=r+1$ (note that $%
j\leq p-1$ since $r\leq p-2$ by assumption). Then $SM(\xi
_{i,j})=x_{r}x_{p}=LM(\varphi _{0})$. Hence $G$ is not reduced.
\end{proof}

\ \ \ \ 

Note the toric varieties in Remark \ref{PnotGBinSomeA3} do not satisfy any
of the conditions \textit{C1} or \textit{C2} as $r=p=r^{\prime }=1$. This
provides a family of toric varieties with reduced Gr\"{o}bner bases, while
the following example provides a mimimal Gr\"{o}bner basis which is not
reduced.

\begin{example}
Let $m_{0}=5,m_{1}=6,m_{2}=7,m_{3}=8,$ and $m_{4}=9$ so that $n=4$ and $p=3$%
. Note $g_{4}-m_{0}=m_{3}+m_{1}-m_{0}=9=m_{4}\in \Gamma $. Hence, $u=4$.
Thus $q=1$ and $r=1$. Thus \textit{C2 holds. }Also, $\upsilon =2$ as $%
2(m_{4})=2m_{0}+m_{3}$. Note $g_{4}=m_{0}+m_{4}$, hence $\lambda =1$ and $%
w=1 $. Also, $\upsilon m_{4}=3m_{0}+m_{3}$, thus $z=3$. Now, $q^{\prime
}=q_{u-z}=0$ and $r^{\prime }=r_{u-z}=1$. Thus C1 holds.
\end{example}

To prove the main theorem we use the following lemma of Aramova et al.

\begin{lemma}
\label{AH}(Lemma 1.1, \cite{AHH}) Let $I\subset R=K[x_{0},\ldots ,x_{n}]$ be
a graded ideal and $G$ a finite subset of homogenous elements of $I$. Given
a term order $<$, there exist a unique monomial $K$-basis $B$ of $%
R/(in_{<}(G))$. If $B$ is a $K$-basis of $R/I$, then $G$ is a Gr\"{o}bner
basis of $I$ with respect to $<$.
\end{lemma}

\begin{remark}
\label{KBasis}Let $P\subset R=K[x_{0},\ldots ,x_{n}]$\ be the kernel of the $%
K$-algebra homomorphism $\eta :R\rightarrow K[t]$ defined by $\eta
(x_{i})=t^{m_{i}}$ with $m_{0},\ldots ,m_{n}$ an almost arithmetic sequence
of positive integers with $gcd(m_{0},\ldots ,m_{n})=1$. Then a set $B$ is a $%
K$-basis of $R/P$ if and only if $l_{1}-l_{2}\notin P$ for any two monomials 
$l_{1},l_{2}\in B$ with $l_{1}\neq l_{2}$.
\end{remark}

\begin{proof}
Assume there exist $l_{1},\ldots ,l_{s}\in B$ and $c_{1},\ldots ,c_{s}\in K$
not all zero such that $\tsum c_{i}l_{i}\in P$. This implies that $\tsum
c_{i}\eta (l_{i})=0$. Hence by the definition of $\eta ,$ there exist $i\neq
j$ such that $\eta (l_{i})=\eta (l_{j})$. This implies that $l_{i}-l_{j}\in
P $.
\end{proof}

\ 

\begin{proof}
\textbf{(of Theorem \ref{MainThm})} Let $G$ be as in the theorem (it
consists of homogenous binomials according to the grading $wt(x_{i})=m_{i}$%
). By Lemma \ref{AH} let $B$ be the unique monomial $K$-basis of $%
R/(in_{<}(G))$. Assume $0\neq l_{1}-l_{2}\in P$ for some monomials $%
l_{1},l_{2}\in B$. Then we show there is a contradiction to Lemma \ref%
{Parameters}, and hence the proof is done by the above lemma and remark.

Throughout the proof let $i,j,$ and $\delta _{k}$ be positive integers such
that $1\leq i,j\leq p-1$ and $\delta _{k}=0$ or $1$. Also, we will use the
sentence "without loss of generality" repeatedly. The usage of this sentence
will be in instances as follows. If a monomial $\beta $ divides $l_{1}$ and $%
l_{2}$, then write $l_{1}-l_{2}=\beta \left( l_{1}^{\prime }-l_{2}^{\prime
}\right) $ with $\beta $ does not divide $l_{1}^{\prime }$ or $\beta $ does
not divide $l_{2}^{\prime }$. Note $l_{1}-l_{2}\in P$ if and only if $\eta
(l_{1})-\eta (l_{2})=0$ if and only if $\eta (l_{1}^{\prime })-\eta
(l_{2}^{\prime })=0$ if and only if $l_{1}^{\prime }-l_{2}^{\prime }\in P$.

First, we work the proof under the assumption that $x_{n}$ divides either $%
l_{1}$ or $l_{2}$. Without loss of generality\ assume $x_{n}^{a_{1}}$
divides $l_{1}$ for some $a_{1}<\upsilon $\ but $x_{n}$ does not divide $%
l_{2}$. Consider two cases:

\underline{Case $x_{0}$ divides neither $l_{1}$ nor $l_{2}$}: then $%
x_{p}^{a_{2}}$ must divide $l_{2}$ for some $a_{2}$, otherwise $l_{2}=x_{j}$
for some $1\leq j\leq p-1$ (as $x_{i}x_{j}\notin B$ for $1\leq i\leq j\leq
p-1$). But this is a contradiction to the minimality of the generating set
of $\Gamma $. We may assume that $x_{p}$ does not divide $l_{1}$, therefore,
we have $l_{1}=x_{j}^{\delta _{1}}x_{n}^{a_{1}}$ and $l_{2}=x_{i}^{\delta
_{2}}x_{p}^{a_{2}}$ with $a_{2}<q+\sigma $ and $\sigma =1$ or $0$ according
as $i<r$ or $i\geq r$. Since $\eta (l_{1})=\eta (l_{2})$ we get the
following equality 
\begin{equation}
\delta _{1}m_{j}+a_{1}m_{n}=\delta _{2}m_{i}+a_{2}m_{p}  \tag{1}  \label{eq1}
\end{equation}%
If $\delta _{1}=0$, then $a_{1}m_{n}\in \Gamma ^{\prime }$, but $%
a_{1}<\upsilon $, thus this gives a contradiction to the minimality of $%
\upsilon $ in Lemma \ref{Parameters}, hence done. Therefore, assume $\delta
_{1}=1$. If $\delta _{2}=0$, then the above equality becomes $%
m_{0}+a_{1}m_{n}=(a_{2}-1)m_{p}+m_{p-j}$. Note the right-hand side is $%
g_{(a_{2}-1)+(p-j)}$. Thus $g_{(a_{2}-1)+(p-j)}-m_{0}=a_{1}m_{n}\in \Gamma $%
. This gives a contradiction to the minimality of $u$ in Lemma \ref%
{Parameters} as $a_{2}-1<q$ and hence $(a_{2}-1)+(p-j)<u$. If $\delta _{2}=1$
then $\left( \ref{eq1}\right) $ becomes $(1-\gamma
)m_{0}+a_{1}m_{n}=(a_{2}-\gamma )m_{p}+m_{\gamma p+i-j}$ with $\gamma =0$ or 
$1$ according as $i>j$ or $i<j$. If $\gamma =1,$ then this gives a
contradiction to the minimality of $\upsilon $, on the other hand, if $%
\gamma =0$, then we get a contradiction to the minimality of $u$ (noting $%
a_{2}\leq q$ if $i<r$\ and $a_{2}<q$ if $i\geq r$).

\underline{Case $x_{0}$ divides either $l_{1}$ or $l_{2}$}:

Consider four subcases:

\underline{Subcase 1}: $x_{0}^{b}$ divides $l_{1}$ for some $b\geq 1$ (and
without loss of generality $x_{0}$ does not divide $l_{2}$). Then $%
x_{p}^{a_{2}}$ must divide $l_{2}$ for some $a_{2}$\ (we may assume that $%
x_{p}$ does not divide $l_{1}$), otherwise $l_{2}=x_{j}$ for some $1\leq
j\leq p-1$ which is a contradiction to the minimality of the generating set
of $\Gamma $. Therefore, we have $l_{1}=x_{0}^{b}x_{j}^{\delta
_{1}}x_{n}^{a_{1}}$ and $l_{2}=x_{i}^{\delta _{2}}x_{p}^{a_{2}}$ with $%
a_{2}<q+\sigma $ and $\sigma =1$ or $0$ according as $i<r$ or $i\geq r$.
Since $\eta (l_{1})=\eta (l_{2})$ we get $bm_{0}+\delta
_{1}m_{j}+a_{1}m_{n}=\delta _{2}m_{i}+a_{2}m_{p}$. This is a contradiction
to the minimality of $u$ as $a_{2}p+i<qp+r=u$ and $b\geq 1$.

\underline{Subcase 2}: $x_{0}^{b}$ divides $l_{2}$ for some $b\geq 1$ (and
without loss of generality $x_{0}$ does not divide $l_{1}$). There are three
subcases;

\underline{Subsubcase 2-1}: $x_{p}$ does not divide any of $l_{1}$ or $l_{2}$%
. Then $l_{1}=x_{j}^{\delta _{1}}x_{n}^{a_{1}}$ and $l_{2}=x_{0}^{b}x_{i}^{%
\delta _{2}}$. Note that if $\delta _{1}=1$, $q^{\prime }=0$, and $a_{1}\geq
\upsilon -w$,\ then we must have $j<r^{\prime }$, otherwise $l_{1}$ is a
multiple of $LM(\psi _{j-r^{\prime }})$ and hence is not in $B$. Since $\eta
(l_{1})=\eta (l_{2})$ we get 
\begin{equation}
\delta _{1}m_{j}+a_{1}m_{n}=bm_{0}+\delta _{2}m_{i}  \tag{2}  \label{eq2}
\end{equation}%
If\ $\delta _{1}=0$, then (\ref{eq2}) becomes $a_{1}m_{n}=bm_{0}+\delta
_{2}m_{i}$. This is a contradiction to the minimality of $\upsilon $. If\ $%
\delta _{1}=1$ and $\delta _{2}=0$, then (\ref{eq2}) becomes $%
m_{j}+a_{1}m_{n}=bm_{0}.$ By Part (iii) and the uniqueness of the parameters
in Lemma \ref{Parameters}, this equality suggests that $a_{1}=\upsilon -w$, $%
\nu =b+1$, and $q^{\prime }=0$. This implies $u-z=j$. But $j<r^{\prime }$ by
the note above, hence $u-z<r^{\prime }$ which is impossible (see Notations %
\ref{W&q'&r'} and \ref{g-sub-t}). If\ $\delta _{1}=1$, $\delta _{2}=1$, and $%
j>i$, then (\ref{eq2}) becomes $m_{j-i}+a_{1}m_{n}=(b+1)m_{0}$. By Part
(iii) and the uniqueness of the parameters in Lemma \ref{Parameters}, this
equality suggests that $a_{1}=\upsilon -w$, $\nu =b+2$, and $q^{\prime }=0$.
This implies $u-z=j-i<r^{\prime }$ which is impossible. If\ $\delta _{1}=1$, 
$\delta _{2}=1$, and $j<i$, then (\ref{eq2}) becomes $%
a_{1}m_{n}=(b-1)m_{0}+m_{i-j}$. This is a contradiction to the minimality of 
$\upsilon $ in Lemma \ref{Parameters}.

\underline{Subsubcase 2-2}: $x_{p}^{a_{2}}$ divides $l_{2}$ for some $a_{2}$%
. Then we have $l_{1}=x_{j}^{\delta _{1}}x_{n}^{a_{1}}$ and $%
l_{2}=x_{0}^{b}x_{i}^{\delta _{2}}x_{p}^{a_{2}}$ with $a_{2}<q+\sigma $ and $%
\sigma =1$ or $0$ according as $i<r$ or $i\geq r$. Since $\eta (l_{1})=\eta
(l_{2})$ we get $\delta _{1}m_{j}+a_{1}m_{n}=bm_{0}+\delta
_{2}m_{i}+a_{2}m_{p}$. Thus $a_{1}m_{n}=(b-\delta _{1})m_{0}+\delta
_{2}m_{i}+(a_{2}-\delta _{1})m_{p}+\delta _{1}m_{p-j}\in \Gamma ^{\prime }$.
This is a contradiction to the minimality of $\upsilon $.

\underline{Subsubcase 2-3}: $x_{p}^{a_{2}}$ divides $l_{1}$ for some $a_{2}$%
. Then we have $l_{1}=x_{j}^{\delta _{1}}x_{p}^{a_{2}}x_{n}^{a_{1}}$ with $%
l_{2}=x_{0}^{b}x_{i}^{\delta _{2}}$ with $\delta _{k}=0$ or $1$ and with
appropriate values of $a_{1},a_{2},i,$ and $j$ so that $l_{1},l_{2}\in B$.
Assume $\delta _{1}=1$. Since $\eta (l_{1})=\eta (l_{2})$ we get $\delta
_{1}m_{j}+a_{2}m_{p}+a_{1}m_{n}=bm_{0}+\delta _{2}m_{i}$. Thus we have 
\begin{equation*}
m_{\gamma p+j-i\delta _{2}}+(a_{2}-\gamma )m_{p}+a_{1}m_{n}=(b+\delta
_{2}-\gamma )m_{0}
\end{equation*}%
where $\gamma =1$ or $0$ according as $i>j$ or $i<j$. By part (iii) and the
uniqueness in Lemma \ref{Parameters}, this equality suggests that $%
u-z=(a_{2}-\gamma )p+\gamma p+j-i\delta _{2}$, $a_{1}=\upsilon -w$, and $\nu
=b+\delta _{2}-\gamma $. This is a contradiction to the uniqueness of $z$
and $\nu $\ since $\delta _{2}$\ and $\gamma $\ may vary non-simultaneously.
Similarly, we get a contradiction for the case $\delta _{1}=0$.

\ Finally, we finish the proof by taking care of the remaining case where $%
x_{n}$ divides neither $l_{1}$ nor $l_{2}$. Consider two cases:

\underline{Case $x_{0}$ divides neither $l_{1}$ nor $l_{2}$}: in such a case 
$l_{1}=x_{i_{1}}^{\delta _{1}}x_{p}^{a_{1}}$ and $l_{2}=x_{i_{2}}^{\delta
_{2}}x_{p}^{a_{2}}$ with $1\leq i_{j}\leq p-1$ and $a_{k}<q+\sigma $ and $%
\sigma =1$ or $0$ according as $i_{k}<r$ or $i_{k}\geq r$. Following similar
process as above one can easily show that there is a contradiction.

\underline{Case $x_{0}$ divides either $l_{1}$ or $l_{2}$}: then, and
without loss of generality, we have $l_{1}=x_{i_{1}}^{\delta
_{1}}x_{p}^{a_{1}}$ and $l_{2}=x_{0}^{b}x_{i_{2}}^{\delta _{2}}x_{p}^{a_{2}}$%
. Following similar process as above one can easily show that there is a
contradiction.
\end{proof}

\ \ \ \ \ 

Patil and Singh \cite{PaSi} constructed a generating set (but not minimal)
for the defining ideal $P$. We call the elements of this set the \textit{%
Patil-Singh generators}. The generators in this set are the same as before
but with different indices as follows (with $q,r,q_{z},r_{z},q^{\prime
},r^{\prime },$ and $\varepsilon $ as before);\ \ \ \ 
\begin{equation*}
\{\xi _{i,j}\mid 1\leq i\leq j\leq p-1\}\cup \{\theta \}\cup \{\varphi
_{i}\mid 0\leq i\leq p-r\}\cup \{\psi _{j}\mid 0\leq j\leq p-r^{\prime }\}%
\text{.}
\end{equation*}

Note that the sets of indices of $\varphi _{i}$ and of $\psi _{i}$ in the
Patil-Singh generators are $[0,p-r]$ and $[0,p-r^{\prime }]$, respectively.
On the other hand, the set of indices of $\varphi _{i}$ and of $\psi _{i}$
in the Patil generators are $I$ and $J$, respectively. It turned out that
the Patil set in contained in $G$ \ (where $G$ as in Theorem \ref{MainThm})
which in turn is contained in the Patil-Singh set. Also, note that the set
of Patil-Singh generators has the advantage of a simpler set of indices than
the set $G$. Therefore, whenever the minimality is not an issue, it is much
easier to deal with the set of Patil-Singh generators than with $G$. The
theorem below proves that the set of Patil-Singh generators is indeed a Gr%
\"{o}bner basis . To prove the theorem below we need the following
proposition which helps to visualize the interval $J$ given by Patil \cite%
{Pat}.

\begin{proposition}
\label{Min_z-1&p-r'}\textit{Let }$z>0$\textit{\ and let }$z=q_{z}p+r_{z%
\mathit{\ }}$\textit{with }$q_{z}\in \mathbb{Z\ }$\textit{and }$r_{z}\in
\lbrack 1,p]$\textit{. Then }$min\{z-1,p-r^{\prime }\}=\left\{ 
\begin{tabular}{ll}
$p-r^{\prime }\text{,}$ & if$\text{ \ }r\leq r_{z}\text{;}$ \\ 
$p-r^{\prime }\text{,}$ & if$\text{ \ }r>r_{z}\text{ and\ }z>p\text{;}$ \\ 
$z-1\text{,}$ & $\text{if \ }r>r_{z}\text{ and }z\leq p\text{.}$%
\end{tabular}%
\right. $\newline
\textit{Moreover, }$z\leq p$\textit{\ if and only if }$q_{z}=0$\textit{.}
\end{proposition}

\begin{proof}
First note that $p-r^{\prime }=(1-\varepsilon )p+r_{z}-r$ where $\varepsilon
=0$ or $1$ according as $r>r_{z}$ or $r\leq r_{z}$. It is obvious that if $%
z>0$\ then $q_{z}\geq 0$. Consider three cases:\newline
Case $r\leq r_{z}$: since $r\in \lbrack 1,p]$\ then $z-1=q_{z}p+r_{z}-1\geq
r_{z}-1\geq r_{z}-r=p-r^{\prime }$.\newline
Case $r>r_{z}$\ and $z>p$: this implies $q_{z}\geq 1.$\ Therefore, $%
z-1=q_{z}p+r_{z}-1\geq p+r_{z}-1\geq p+r_{z}-r=p-r^{\prime }$.\newline
Case $r>r_{z}$\ and $z\leq p$: this implies $q_{z}=0$. Therefore, $%
z-1=r_{z}-1\leq r_{z}-1+p-r<p+r_{z}-r=p-r^{\prime }$.
\end{proof}

\ \ \ \ 

Therefore, whenever $W\neq \phi $ we write $J$\ as follows

$J=\left\{ 
\begin{tabular}{ll}
$\lbrack 0,p-r^{\prime }]\text{,}$ & if $q_{z}>0$ or $\varepsilon >0$; \\ 
$\lbrack 0,r_{z}-1]\text{,}$ & if $q_{z}=0$ and $\varepsilon =0$.%
\end{tabular}%
\right. $

\begin{theorem}
\label{SecondMainThm}\textit{The set }$S=\{\varphi _{i}\mid 0\leq i\leq
p-r\}\cup \mathit{\ }\{\psi _{j}\mid 0\ \leq \ j\leq p-r^{\prime }\}\cup
\{\theta \}$\textit{\ }$\cup $\textit{\ }$\{\xi _{i,j}\mid 1\leq i\leq j\leq
p-1\}$, that is, the set of Patil-Singh generators,\textit{\ is a Gr\"{o}%
bner basis (not minimal) for the ideal }$P$\textit{\ with respect to the
grevlex monomial order with }$x_{0}<x_{1}<\cdots <x_{n}$\textit{\ and with
the grading }$wt(x_{i})=m_{i}$\textit{.}\ \ 
\end{theorem}

\begin{proof}
Recall $q^{\prime }=q-q_{z}-\varepsilon $ and $r^{\prime }=\varepsilon
p+r-r_{z}$ where $\varepsilon =0$ or $1$ according as $r>r_{z}$ or $r\leq
r_{z}$. If $q_{z}>0$ or $\varepsilon >0$, then $J=[0,p-r^{\prime }]$ and the
set of Patil-Singh generators coincides with the set $G$ of Theorem \ref%
{MainThm}$,$ hence done. If $q_{z}=0$ and $\varepsilon =0$, then $q^{\prime
}=q$, $r^{\prime }=r-r_{z}$, and$\ J=[0,r_{z}-1]$. Also note $r_{z}\leq
r_{z}+p-r=p-r^{\prime }$. Now consider $LM(\psi _{j})$ where $j$ runs over $%
[r_{z},p-r^{\prime }]$ (this indicates the binomials that exist in
Patil-Singh but not in $G$) we get $\{LM(\psi _{j})=x_{j+r^{\prime
}}x_{p}^{q^{\prime }}x_{n}^{\upsilon -w}$ $\mid r_{z}\ \leq $\ $j\leq
p-r^{\prime }\}=\{x_{j}x_{p}^{q}x_{n}^{\upsilon -w}$ $\mid r\ \leq $\ $j\leq
p\}=x_{n}^{\upsilon -w}\{LM(\varphi _{i})=x_{j+r}x_{p}^{q}$ $\mid 0\ \leq $\ 
$j\leq p-r\}$.\ Therefore, the monomial $K$-basis of $R/(in_{<}(S))$ is
essentially the same as the monomial $K$-basis of $R/(in_{<}(G))$ where $S$
is the set of the Patil-Singh generators. Hence done by Lemma \ref{AH}.
\end{proof}

\ \ \ \ \ \ 

Finally, we finish this paper by noting that Patil-Singh generators do not
form a Gr\"{o}bner basis in all cases if we consider the grevlex monomial
order with the same grading as before but with\ $x_{0}>x_{1}>\cdots >x_{n}$
( in this case \textit{\ }$\prod\limits_{i=0}^{n}x_{i}^{a_{i}}>_{grevlex}%
\prod\limits_{i=0}^{n}x_{i}^{b_{i}}$\ if in the ordered tuple $%
(a_{1}-b_{1},\ldots ,a_{n}-b_{n})$\ the right-most nonzero entry is
negative). In the following we prove this and give an example.

\begin{remark}
\label{P-SnotGB}Assume $r<r_{z}<p$ (hence $\varepsilon =0$), $\lambda >1$,
and $w>0$. Then \textit{Patil-Singh generators are not a Gr\"{o}bner basis
with respect to the grevlex monomial ordering with }$x_{0}>x_{1}>\cdots
>x_{n}$ \textit{and \ with the grading }$wt(x_{i})=m_{i}$\textit{\ .}
\end{remark}

\begin{proof}
First note $LT(\varphi _{i})=x_{i+r}x_{p}^{q}$ if $w>0$ and $LT(\varphi
_{i})=x_{0}^{\lambda -1}x_{i}$\ if $w=0$. Also, $LT(\psi
_{j})=x_{0}^{\lambda +\mu -\varepsilon }x_{j}$, $LT(\theta )=x_{0}^{\mu
}x_{r_{z}}x_{p}^{q_{z}}$, and $LT\left( \xi _{i,j}\right) =x_{i}x_{j}$. If $%
r<r_{z}<p$ (hence $\varepsilon =0$), $\lambda >1$, and $w>0$, then none of
the terms of $S(\xi _{1,r_{z}},\theta )=x_{1}x_{n}^{\upsilon }-x_{0}^{\mu
+1}x_{r_{z}+1}x_{p}^{q_{z}}$ is a multiple of any of the leading terms of
the Patil-Singh generators.
\end{proof}

\begin{example}
Let $m_{0}=20,m_{1}=21,m_{2}=22,m_{3}=23,m_{4}=24,$ and $m_{5}=29$. Note $%
n=5 $ and $p=4$. Let $P$\ be the kernel of the $K$-algebra homomorphism $%
\eta :K[x_{0},\ldots ,x_{5}]\rightarrow K[t]$ defined by $\eta
(x_{i})=t^{m_{i}}$. Recall the parameters in Lemma~\ref{Parameters}. It is
easy to check that $\upsilon =3$, hence by the uniqueness condition we must
have $\mu =2$, $q_{z}=1$, and $r_{z}=3$, thus $z=7$. For $1\leq i\leq 3$
note that in order for $am_{4}+m_{i}-m_{0}$ to be in $\ \Gamma $ we must
have $a\geq 2$. Note $g_{2p+1}=2(24)+21=2(20)+29$. Therefore, we conclude
that $u=2p+1=9$, thus $q=2$ and $r=1$ . Hence, $\lambda =2$, $w=1$, $%
r^{\prime }=2$, and $q^{\prime }=1$. Therefore, Patil-Singh generators are
as follows: $G=\{\varphi _{i}\mid 0\leq i\leq 3\}\cup \{\psi _{j}\mid 0\leq
j\leq 2\}\cup \{\theta \}\cup \{\xi _{i,j}\mid 1\leq i\leq j\leq 3\}$ where $%
\varphi _{i}=\underline{x_{i+1}x_{4}^{2}}-x_{0}x_{i}x_{5}$, and $\psi
_{j}=x_{j+2}x_{5}^{2}-\underline{x_{0}^{3}x_{j}}$, and $\theta =x_{5}^{3}-%
\underline{x_{0}^{2}x_{3}x_{4}}$ and $\xi _{i,j}=\underline{x_{i}x_{j}}%
-x_{0}^{(1-\gamma )}x_{i+j-\gamma p}x_{p}^{\gamma }$ with $\gamma =0$ or $1$
according as $i+j\leq p$ or $i+j>p$. The set $G$ is not Gr\"{o}bner basis
with respect to the grevlex monomial ordering with $x_{0}>x_{1}>\cdots
>x_{5} $ and with the grading $wt(x_{i})=m_{i}$: consider $S(\theta ,\xi
_{1,3})=x_{1}x_{5}^{3}-x_{0}^{3}x_{4}^{2}$. Note that neither term of $%
S(\theta ,\xi _{1,3})$ is a multiple of any of the leading terms above.%
\vspace{0in}
\end{example}

\section*{Acknowledgement}

The author thanks Professor Irena Swanson for the useful discussions and
comments during the course of this work. Also, the author thanks the referee
for the very useful suggestion that simplified the proof much easier than
the original form.

\ \ \ \ 

{\small Ibrahim Al-Ayyoub, assistant professor}

{\small Department of Mathematics and Statistics}

{\small Jordan University of Science and Technology}

{\small P O Box 3030, Irbid 22110, Jordan.}

{\small Email address: iayyoub@just.edu.jo}

{\small \ \ \ \ \ \ \ \ }

\end{document}